\documentclass{amsproc}

\newtheorem{theorem}{Theorem}[section]

\newtheorem{proposition}[theorem]{Proposition}
\newtheorem{corollary}[theorem]{Corollary}
\theoremstyle{definition}

\theoremstyle{remark}
\newtheorem{remark}[theorem]{Remark}

\numberwithin{equation}{section}
\usepackage{amsthm,amsmath,amssymb,amscd,amsfonts,bm, enumerate}
\begin{document}
\title[Ideal Generation from Vanishing Theorem]{Skoda's Ideal Generation from Vanishing Theorem for Semipositive Nakano Curvature and Cauchy-Schwarz Inequality for Tensors}
\author[Y-T. Siu]{Yum-Tong Siu}
\address{Department of Mathematics, Harvard University}
\email{siu@math.harvard.edu}
\date{January 2018}
\dedicatory{Dedicated to Lawrence Ein}
\begin{abstract}The Bochner-Kodaira technique of completion of squares yields vanishing theorems and $L^2$ estimates of $\bar\partial$.  Skoda's ideal generation, which is a crucial ingredient in the analytic approach to the finite generation of the canonical ring and the abundance conjecture, requires specially tailored analytic techniques for its proof.  We introduce a new method of deriving Skoda's ideal generation which makes it and its formulation a natural consequence of the standard techniques of vanishing theorems and $L^2$ estimates of $\bar\partial$. Our method of derivation readily gives also other similar results on ideal generation.  An essential r\^ole is played by one particular Cauchy-Schwarz inequality for tensors with a special factor which accounts for the exponent of the denominator in the formulation of the integral condition for Skoda's ideal generation.\end{abstract}
\maketitle

\section*{Introduction} Since the nineteen sixties the Bochner-Kodaira technique of completion of squares, for manifolds either compact or with pseudoconvex boundary condition, has been a most important tool in algebraic and complex geometry. Sometimes ingenious {\it ad hoc} adaptations are required in its use.  A remarkable example is the 1972 result of Skoda on ideal generation (Th\'eor\`eme 1 on pp.555--556 of \cite{Skoda1972}), which is a crucial ingredient in the analytic approach to the finite generation of the canonical ring and the abundance conjecture.  Up to this point, special analytic techniques developed by Skoda, other than applications of the usual vanishing theorems and $L^2$ estimates for the $\bar\partial$-equation, are required for its proof.

In recent years analytic results from vanishing theorems and solvability of $\bar\partial$-equation have contributed to the solution of a number of longstanding open problems in algebraic geometry.  The interaction between algebraic geometry and analytic methods in several complex variables and partial differential equations has been a very active and productive area of investigation.  For such an interaction it is advantageous to minimize the use of {\it ad hoc} analytic methods in favor of approaches which are more amenable to adaptation to formulations in algebraic geometry.

In this note we give a simpler, more straightforward proof of Skoda's result on ideal generation which makes it and its formulation a natural consequence of the standard techniques in vanishing theorems and solving $\bar\partial$-equation with $L^2$ estimates. Our proof
readily gives other similar results on ideal generation.  In \S\ref{section6:VariantsOfSkodasIdealGeneration} we present a number of such similar results. Our proof uses the following three ingredients.

\subsection*{}(I) One particular Cauchy-Schwarz inequality for tensors
$$
\left|\left<{\mathbf S},\,{\mathbf T}\right>_{{\mathbb C}^r\otimes{\mathbb C}^n}\right|^2
\leq\min(r,n)\left\|\left<{\mathbf S},\,{\mathbf T}\right>_{{\mathbb C}^n}\right\|_{{\mathbb C}^r\otimes{\mathbb C}^r}^2
$$
for elements ${\mathbf S}$ and ${\mathbf T}$ of ${\mathbb C}^r\otimes{\mathbb C}^n$ with the special factor $\min(r,n)$. Detailed statement and proof will be given in \S\ref{section2:Cauchy-SchwarzInequalityForTensors}. The usual Cauchy-Schwarz inequality is the special case
of $n=1$.  Our particular Cauchy-Schwarz inequality for tensors {\it with the factor $\min(r,n)$ on the right-hand side} is needed in order to handle the curvature operator condition for nonnegative Nakano curvature.  Just like the usual Cauchy-Schwarz inequality, it is derived by a straightforward simple reduction to certain special orthonormal situation.  The significance of our Cauchy-Schwarz inequality is the factor $\min(r,n)$ which plays an essential r\^ole the exponent of the denominator in the formulation of the integral condition for Skoda's ideal generation.  Though ${\mathbb C}^r\otimes{\mathbb C}^n$ can be identified with ${\mathbb C}$-linear operators between the vector spaces ${\mathbb C}^r$ and ${\mathbb C}^n$ and there are generalizations of the usual Cauchy-Schwarz inequality to bounded ${\mathbb C}$-linear operators between Hilbert spaces, yet the nature of existent generalizations is different and the essential factor $\min(r,n)$ does not occur readily in them without repeating the kind of arguments used in deriving our Cauchy-Schwarz inequality (see Remark \ref{remark:2.3}).  Our Cauchy-Schwarz inequality for tensors should be a simple case of more general inequalities of Cauchy-Schwarz type for general tensors with special factors (obtained by reduction to orthonormal situations in Young tableau).  In \S\ref{section7:ComparisonWithOriginalProofOfSkodasIdealGeneration} where the manipulations in computation between our proof and Skoda's original proof are compared, one such inequality of Cauchy-Schwarz type for more general tensors is given.

\subsection*{}(II) Nonnegativity of the Nakano curvature of the metric for the kernel bundle of $(g_1,\cdots,g_p):\left({\mathcal O}_\Omega\right)^{\oplus p}\to {\mathcal O}_\Omega$ which is induced from the standard flat metric of $\left({\mathcal O}_\Omega\right)^{\oplus p}$ and twisted by the weight function $\left(\sum_{j=1}^q|g_j|^2\right)^{-q}$, where $\Omega$ is a domain in ${\mathbb C}^n$ and $q=\min(n,p-1)$.  This is verified by a simple straightforward computation in normal fiber coordinates at the point under consideration and the use of the above Cauchy-Schwarz inequality for tensors.   Detailed statement and proof will be given in \S\ref{section3:NakanoCurvatureOfBundleAndTwistingByMetricOfTrivialLineBundle}.

\subsection*{}(III) Vanishing theorem (from solving a $\bar\partial$-equation with $L^2$ estimate) for a holomorphic vector bundle $V$ with nonnegative Nakano curvature $\Theta$ on a strictly pseudoconvex domain $\Omega$ in an ambient complex manifold of complex dimension $n$ for smooth $\bar\partial$-closed $V$-valued $(n,1)$-form $F$ with finite $\left(\Theta^{-1}F,\,F\right)_{L^2(\Omega)}$, where $\Theta^{-1}$ is defined as the limit, as the positive number $\varepsilon$ approaches $0$, of the inverse of $\Theta+\varepsilon I$ with $I$ being the identity operator.  Actually what is needed is only the special case where $X$ is a Stein domain spread over ${\mathbb C}^n$ where the  canonical line bundle $K_X$ is trivial and $F$ can be replaced by a smooth $\bar\partial$-closed $V$-valued $(0,1)$-form.  Detailed statement and proof will be given in \S\ref{section4:VanishingTheoremAndSolutionOfDBarEquationForNonnegativeNakanoCurvature}.  Usually an application of the technique of vanishing theorems to produce holomorphic sections indispensably requires strict positivity of curvature at least at some point.  The special feature of the application here of the vanishing theorem from solving a $\bar\partial$-equation with $L^2$ estimate requires only semipositive Nakano curvature and {\it no strict positivity at any point}. The difficulty of producing holomorphic sections by solving $\bar\partial$-equation with $L^2$ estimate without positivity of curvature at any point is that it is not possible to get a usable right-hand side for the $\bar\partial$-equation by applying $\bar\partial$ to the product of a cut-off function and a local holomorphic section.  For the problem of ideal generation, the nonnegative curvature condition is sufficient for the application, because there is a natural usable choice for the right-hand side of the $\bar\partial$-equation.  Another known result requiring only nonnegative curvature condition is the extension result of Ohsawa-Takegoshi.
Applications to produce holomorphic sections from only nonnegative curvature condition remain mostly an area yet to be explored.

\subsection*{\it Notations.} The set of all positive integers is denoted by ${\mathbb N}$.  We will use the notation ${\mathcal O}_X$ to denote the structure sheaf of $X$.  To minimize the use of notations we will loosely refer to the locally free sheaf associated to a holomorphic vector bundle simply as a holomorphic vector bundle and vice versa.  For example, we use the notation $\left({\mathcal O}_{{\mathbb C}^n}\right)^{\oplus p}$ to denote both the globally trivial vector bundle of rank $p$ over ${\mathbb C}^n$ and the direct sum of $p$ copies of the structure sheaf ${\mathcal O}_{{\mathbb C}^n}$ of ${\mathbb C}^n$.  We use $\Gamma(Y,W)$ (or $\Gamma(Y,{\mathcal F})$) to denote the space of holomorphic sections of a bundle $W$ (or a coherent sheaf ${\mathcal F}$) over $Y$.

For a holomorphic vector bundle $V$ of rank $r$ with a metric $H_{j\bar k}$ over a complex manifold $X$ of complex dimension $n$ with a K\"ahler metric $g_{\lambda\bar\nu}$, we will use the Latin letters $j,k,i,\ell$ etc. for the indices of the fiber coordinates of $V$ and use the Greek letters $\lambda,\nu,\mu,\rho$ etc. for the local coordinates $z^1,\cdots,z^n$ of $X$.  We use $g^{\bar\nu\lambda}$ to mean the inverse of $g_{\lambda\bar\nu}$ and use $H^{\bar kj}$ to mean the inverse of $H_{j\bar k}$.  Unless there is a possible confusion, without explicit mention we will use $H^{\bar kj}$ and $g^{\bar\nu\lambda}$ to raise the subscripts of tensors and use $H_{j\bar k}$ and $g_{\lambda\bar\nu}$ to lower the superscripts of tensors.  The notation $T_X$ denotes the (holomorphic) tangent bundle of $X$.  The notation $K_X$ denotes the canonical line bundle of $X$.

For a holomorphic subbundle $W$ of $V$, when we use the metric of $W$ which is induced from the metric $H_{j\bar k}$ of $V$, we simply say the metric $H_{j\bar k}$ of $W$ instead of introducing a new notation.

When we refer to a positive $(1,1)$-form $\sqrt{-1}\sum_{1\leq\lambda,\nu\leq n}\eta_{\lambda\bar\nu}dz^\lambda\wedge d\overline{z^\nu}$, we will drop the factor $\sqrt{-1}$ and simply say that the $(1,1)$-form $\eta_{\lambda\bar\nu}$ is positive.

The notation $\partial_\lambda$ means $\frac{\partial}{\partial z^\lambda}$ and the notation $\partial_{\bar\lambda}$ means $\frac{\partial}{\partial\overline{z^\lambda}}$. The notation $\left(\cdot,\,\cdot\right)$ or $\left<\cdot,\,\cdot\right>$ means the inner product and the notation $\left\|\cdot\right\|$ means the norm, sometimes with a subscript to indicate in which space the inner product or the norm is taken.

The study of vanishing theorems and $L^2$ estimates for the $\bar\partial$-equation has a very long history and a very extensive literature.  To avoid a long bibliography which nowadays can easily be fetched from readily available online searchable databases, we keep to a minimum the listing of references at the end of this note.

\section{Algebraic Formulation and Application of Skoda's Ideal Generation}\label{section1:AlgebraicFormulationAndApplicationOfSkodasIdealGeneration}

Though this is not part of our proof of Skoda's result on ideal generation, to draw attention to its use in algebraic geometric problems, we give here the algebraic geometric formulation of
Skoda's result and how it is used in the analytic approach to the finite generation of the canonical ring before we present our proof.

We state in the following way a trivially more general form of Skoda's result than is given in Th\'eor\`eme 1 on pp.555--556 of \cite{Skoda1972} by using a Stein domain spread over ${\mathbb C}^n$ instead of a Stein domain in ${\mathbb C}^n$, in order to be able to formulate it in an algebraic geometric setting.

\begin{theorem}[Skoda's Theorem on Ideal Generation]\label{theorem:1.1}
Let $\Omega $ be a domain spread over ${\mathbb C}^n$ which is
Stein. Let $\psi$ be a plurisubharmonic function on $\Omega$,
$g_{1},\ldots ,g_{p}$ be holomorphic functions on $\Omega$ (not all identically zero), $\alpha
>1$, $q=\min \left(n,p-1\right)$, and $|g|^2=\sum_{j=1}^p\left| g_j\right|^2$.
Then for any holomorphic function $f$ on $\Omega$ with
$$\int_{\Omega }\frac{\left\vert f\right|^{2}e^{-\psi }}{|g|^{2(\alpha q+1)}}<\infty,$$
there exist holomorphic functions $h_{1},\ldots ,h_{p}$ on
$\Omega$ with $f=\sum_{j=1}^{p} h_jg_j$ on $\Omega$ such that
$$\int_{\Omega }\frac{\sum_{k=1}^p\left|h_{k}\right|^2e^{-\psi
}}{|g|^{2\alpha q}}\leq
\frac{\alpha }{\alpha -1}\int_{\Omega
}\frac{\left|f\right|^2e^{-\psi }}{|g|^{2(\alpha q+1)}}
$$
holds.
\end{theorem}

The following algebraic geometric formulation of Skoda's result is easily obtained by representing the compact complex algebraic manifold $X$ of complex dimension $n$ minus some complex hypersurface as a Stein domain over ${\mathbb C}$ and considering a holomorphic line bundle over $X$ minus some complex hypersurface as globally trivial by division by a meromorphic section of it.

\begin{theorem}[Algebraic Geometric Formulation of Ideal Generation]\label{theorem:1.2}
Let $X$ be a compact complex algebraic manifold of complex dimension $n$ and $L$ and $E$ be respectively
holomorphic line bundles on $X$ with (possibly singular) Hermitian metrics $e^{-\varphi_L}$ and $e^{-\varphi_E}$ such
that $\varphi_L$ and $\varphi_E$ are plurisubharmonic. Let $s_1,\cdots,s_p\in\Gamma(X,L)$ and $0<\gamma\leq 1$ and $q=\max(n,p-1)$.  Then
for any $s\in\Gamma(X,(q+2)L+E+K_X)$ with
$$
C_s:=\int_X\frac{|s|^2e^{-(1-\gamma)\varphi_L-\varphi_E}}{\left(\sum_{j=1}^p|s_j|^2\right)^{q+1+\gamma}}<\infty,
$$
there exist $h_1,\cdots,h_p\in\Gamma(X,((q+1)L+E+K_X)$ such that $s=\sum_{j=1}^p h_js_j$ with
$$
\int_X\frac{\sum_{k=1}^p\left|h_k\right|^2e^{-(1-\gamma)\varphi_L-\varphi_E}}{\left(\sum_{j=1}^p|s_j|^2\right)^{q+\gamma}}\leq\left(1+\frac{q}{\gamma}\right)C_s.
$$
\end{theorem}

\begin{theorem}[Finite Generation of Section Module]  Let $L$ be a holomorphic line bundle over a compact complex algebraic manifold $X$ of complex dimension $n$.  Let $s^{(m)}_1,\cdots,s^{(m)}_{q_m}$ be a ${\mathbb C}$-basis of $\Gamma(X,mL)$.  Let $N_0$ and $m_0$ be positive integers such that $N_0\geq m_0(n+2)$ and $\Gamma(X,m_0L)\not=0$.  Let $\varepsilon_m$ (for $m\in{\mathbb N}$) be a sequence of positive numbers which decrease rapidly enough to yield the convergence of
the infinite series
$$
\Phi:=\sum_{m=1}^\infty\varepsilon_m\sum_{j=1}^{q_m}\left|s_j^{(m)}\right|^{\frac{2}{m}}
$$
on $X$.  Suppose there exist elements $s^{(m_0)}_1,\cdots,s^{(m_0)}_{q_{m_0}}$ of $\Gamma(X,m_0L)$ such that for any $m\geq N_0$ and for any $s^{(m)}\in\Gamma(X, mL+K_X)$ the (locally defined) function
\begin{equation}\label{eq:1.3.1FiniteGenerationOfSectionModule1}
\frac{\left|s^{(m)}\right|^2}{\Phi^{{}^{m-(n+2)m_0}}\left(\sum_{j=1}^{q_{m_0}}\left|s^{(m_0)}_j\right|^2\right)^{n+2}}
\end{equation}
is locally bounded on $X$.  Then for $m\geq N_0$,
$$
\Gamma(X,mL+K_X)=\sum_{\nu_1+\cdots+\nu_{q_{m_0}}=\ell_m}\left(s^{(m_0)}_1\right)^{\nu_1}\cdots\left(s^{(m_0)}_{q_{m_0}}\right)^{\nu_{q_{m_0}}}\Gamma\left((m-m_0\ell_m)L+K_X\right),
$$
where $\ell_m$ the largest integer with $m-\ell_m m_0\geq N_0$.

In particular, if one denotes by $R(X,F)$ the ring $\bigoplus_{k=1}^\infty\Gamma(X,kF)$ of all global holomorphic sections of positive tensor powers of a line bundle $F$ over $X$, then the module
$\bigoplus_{m=1}^\infty\Gamma(X,mL+K_X)$ is generated by a finite number of elements over the ring $R(X,m_0L)$ and the finite number of elements can be taken
to be a ${\mathbb C}$-basis of the finite-dimensional ${\mathbb C}$-vector space $\bigoplus_{m=1}^{N_0}\Gamma(X,mL+K_X)$.  In other words,
$$\bigoplus_{m=1}^\infty\Gamma(X,mL+K_X)=R(X,m_0L)\left(\bigoplus_{m=1}^{N_0}\Gamma(X,mL+K_X)\right).$$

In the special case of $L=K_X$, under the condition of the local boundedness of (\ref{eq:1.3.1FiniteGenerationOfSectionModule1}) the canonical ring $R(X,K_X)$ is finitely generated.
\end{theorem}

\begin{proof} Let $e^{-\varphi_L}=\frac{1}{\Phi}$.
For $m\geq N_0$ and any $s\in\Gamma\left(X, mL+K_X\right)$, by Theorem \ref{theorem:1.2} from
$$
\int_X\frac{\left|s\right|^2e^{-\left(m-\left(n+2\right)m_0\right)\varphi_L}}
{\left(\sum_{j=1}^{q_{m_0}}\left|s^{\left(m_0\right)}_j\right|^2\right)^{n+2}}<\infty,
$$
it follows that there
exist
$$
h_1,\cdots,h_{q_{m_0}}\in\Gamma\left(X,\left(m-m_0\right)L\right)
$$
such that $s=\sum_{j=1}^{q_{m_0}}h_j s^{\left(m_0\right)}_j$.   If
$m-\left(n+2\right)m_0$ is still not less than $N_0$,
we can apply the argument to each $h_j$ instead of $s$ until by induction on $\nu$ we get
$$
h^{(j_1,\cdots,j_\nu)}_1,\cdots,h^{(j_1,\cdots,j_\nu)}_{q_{m_0}}\in\Gamma\left(X,\left(m-m_0\left(\nu+1\right)\right)L+K_X\right)
$$
for $1\leq j_1,\cdots,j_\nu\leq q_{m_0}$ with $1\leq\nu\leq\ell_m$ such that
$$
s=\sum_{1\leq j_1,\cdots, j_\nu\leq
q_{m_0}}h^{(j_1,\cdots,j_{\nu-1})}_{j_\nu}\prod_{\lambda=1}^\nu s^{\left(m_0\right)}_{j_\lambda}
$$
with $h^{(0)}_j=h_j$ for $1\leq j\leq q_{m_0}$.
\end{proof}

We now introduce the three ingredients for our proof of Skoda's ideal generation.

\section{Cauchy-Schwarz Inequality for Tensors}\label{section2:Cauchy-SchwarzInequalityForTensors}

The following Cauchy-Schwarz inequality for tensors is introduced to handle inequalities for curvature operator conditions.  Its proof will be done by applying the usual Cauchy-Schwarz inequality for vectors and a reduction to the special case of an orthonormal set of vectors by linear transformations.

\begin{proposition}[Cauchy-Schwarz Inequality for Tensors]\label{Proposition:2.1CauchySchwarzInequalityForTensors}  Let ${\mathbf S}$ and ${\mathbf T}$ be elements of ${\mathbb C}^r\otimes{\mathbb C}^n$.  Then
$$
\left|\left<{\mathbf S},\,{\mathbf T}\right>_{{\mathbb C}^r\otimes{\mathbb C}^n}\right|^2
\leq\min(r,n)\left\|\left<{\mathbf S},\,{\mathbf T}\right>_{{\mathbb C}^n}\right\|_{{\mathbb C}^r\otimes{\mathbb C}^r}^2.
$$
\end{proposition}

\begin{remark}\label{remark:2.2}  When the tensors ${\mathbf S}$ and ${\mathbf T}$ are represented by $r\times n$ matrices
$${\mathbf S}=\left(s_{\ell,k}\right)_{1\leq\ell\leq r, 1\leq k\leq n}\quad{\rm and}\quad{\mathbf T}=\left(t_{\ell,k}\right)_{1\leq\ell\leq r, 1\leq k\leq n},
$$
the inequality in Proposition \ref{Proposition:2.1CauchySchwarzInequalityForTensors} becomes
$$
\left|\sum_{1\leq\ell\leq r,1\leq k\leq n}
s_{\ell,k}t_{\ell,k}
\right|^2\leq\min(r,n)\sum_{1\leq m,\ell\leq r}
\left|\sum_{k=1}^n
s_{m,k}t_{\ell,k}
\right|^2.
$$
The special case of $n=1$ is simply the following usual Cauchy-Schwarz inequality for vectors
$$
\left|\sum_{1\leq\ell\leq r}
s_\ell t_\ell
\right|^2\leq \left(\sum_{1\leq\ell\leq r}|s_\ell|^2\right)\left(\sum_{1\leq\ell\leq r}|t_\ell|^2\right)
$$
where $s_\ell=s_{\ell,1}$ and $t_\ell=t_{\ell,1}$.
\end{remark}

\begin{proof}[Proof of Proposition \ref{Proposition:2.1CauchySchwarzInequalityForTensors}] Let ${\mathbf S}=\left(s_{\ell,k}\right)_{1\leq\ell\leq r,\,1\leq k\leq n}$ and ${\mathbf T}=\left(t_{\ell,k}\right)_{1\leq\ell\leq r,\,1\leq k\leq n}$.  The inequality
$$
\left|\left<{\mathbf S},\,{\mathbf T}\right>_{{\mathbb C}^r\otimes{\mathbb C}^n}\right|^2
\leq r\left\|\left<{\mathbf S},\,{\mathbf T}\right>_{{\mathbb C}^n}\right\|_{{\mathbb C}^r\otimes{\mathbb C}^r}
$$
is the same as
$$
\left|\sum_{1\leq\ell\leq r,1\leq k\leq n}
s_{\ell,k}t_{\ell,k}
\right|^2\leq r\sum_{1\leq m,\ell\leq r}
\left|\sum_{k=1}^n
s_{m,k}t_{\ell,k}
\right|^2,
$$
which simply follows from the following application of the Cauchy-Schwarz inequality with focus on the summation over the double index $(\ell,m)$ (where $\delta_{\ell m}$ for $1\leq\ell, m\leq r$ is the Kronecker delta)
$$
\begin{aligned}
&\left|\sum_{1\leq\ell\leq r,1\leq k\leq n}
s_{\ell,k}t_{\ell,k}
\right|^2\cr
&=\left|\sum_{1\leq\ell,m\leq r,1\leq k\leq n}
\delta_{\ell m}s_{m,k}t_{\ell,k}
\right|^2\cr
=&\left|\sum_{1\leq\ell,m\leq r}
\delta_{\ell m}\left(\sum_{k=1}^n s_{m,k}t_{\ell,k}\right)
\right|^2\cr
&\leq\left(\sum_{1\leq m,\ell\leq r}(\delta_{m\ell})^2\right)\left(\sum_{1\leq m,\ell\leq r}
\left|\sum_{k=1}^n
s_{m,k}t_{\ell,k}
\right|^2\right)\cr
&=r\sum_{1\leq m,\ell\leq r}
\left|\sum_{k=1}^n
s_{m,k}t_{\ell,k}
\right|^2.\cr
\end{aligned}
$$
The inequality
$$
\left|\left<{\mathbf S},\,{\mathbf T}\right>_{{\mathbb C}^r\otimes{\mathbb C}^n}\right|^2
\leq n\left\|\left<{\mathbf S},\,{\mathbf T}\right>_{{\mathbb C}^n}\right\|_{{\mathbb C}^r\otimes{\mathbb C}^r}
$$
is the same as
$$
\left|\sum_{1\leq\ell\leq r,1\leq k\leq n}
s_{\ell,k}t_{\ell,k}
\right|^2\leq n\sum_{1\leq m,\ell\leq r}
\left|\sum_{k=1}^n
s_{m,k}t_{\ell,k}
\right|^2.
$$
Instead of proving this inequality we prove the following equivalent inequality which is obtained by replacing $t_{\ell,k}$ by its complex conjugate $\overline{t_{\ell,k}}$.
\begin{equation}\label{eq:2.3.1PropositionCauchySchwarzInequalityForTensors}
\left|\sum_{1\leq\ell\leq r,1\leq k\leq n}
s_{\ell,k}\overline{t_{\ell,k}}
\right|^2\leq n\sum_{1\leq m,\ell\leq r}
\left|\sum_{k=1}^n
s_{m,k}\overline{t_{\ell,k}}
\right|^2.
\end{equation}
Denote the $n$ column $r$-vectors of the $r\times n$ matrix $${\mathbf S}=\left(s_{\ell,k}\right)_{1\leq\ell\leq r, 1\leq k\leq n}$$ by $\vec s_1,\cdots,\vec s_n$.  Likewise we denote the $n$ column $r$-vectors of the $r\times n$ matrix $${\mathbf T}=\left(t_{\ell,k}\right)_{1\leq\ell\leq r, 1\leq k\leq n}$$ by $\vec t_1,\cdots,\vec t_n$.

To prove (\ref{eq:2.3.1PropositionCauchySchwarzInequalityForTensors}) it suffices to assume that $n\leq r$ and by slightly perturbing $\vec s_1,\cdots,\vec s_n$ and then taking the limit to go back to the original set of $\vec s_1,\cdots,\vec s_n$, we can assume without loss of generality that the vectors $\vec s_1,\cdots,\vec s_n$ of ${\mathbb C}^r$ are ${\mathbb C}$-linearly independent.  We apply a ${\mathbb C}$-linear transformation given by an $n\times n$ matrix $R$ to the $n$ elements
${\vec t}_1,\cdots,{\vec t}_n$
of ${\mathbb C}^r$ and at the same time the ${\mathbb C}$-linear transformation given by the $n\times n$ matrix $\overline{\left(R^t\right)^{-1}}$, which is the complex-conjugate transpose inverse of $R$, to
the $n$ elements
${\vec t}_1,\cdots,{\vec t}_n$
of ${\mathbb C}^r$ in order to preserve the inner product
$$\sum_{k=1}^n s_{j,k}\overline{t_{\ell,k}}$$
for any fixed $1\leq j,\ell\leq r$.  So without loss of generality we can assume that the set
of $n$ elements
${\vec s}_1,\cdots,{\vec s}_n$
of ${\mathbb C}^r$ are orthonormal with respect to $\left<\cdot,\,\cdot\right>_{{\mathbb C}^r}$ and explicitly given by $s_{j,k}$ being equal to the Kronecker delta $\delta_{jk}$ for $1\leq j,k\leq n$ and equal to $0$ for $n<j\leq r$.  In particular, the expression
$$
\sum_{1\leq m,\ell\leq r}
\left|\sum_{k=1}^n
s_{m,k}\overline{t_{\ell,k}}
\right|^2
$$
is equal to
$$
\sum_{\ell=1}^r\sum_{m=1}^n
\left|\sum_{k=1}^n
\delta_{mk}\overline{t_{\ell,k}}
\right|^2=\sum_{\ell=1}^r\sum_{m=1}^n
\left|\overline{t_{\ell,m}}
\right|^2=\sum_{k=1}^n\left\|\vec t_k\right\|^2_{{\mathbb C}^r}.
$$
The left-hand of the inequality (\ref{eq:2.3.1PropositionCauchySchwarzInequalityForTensors}) can be rewritten as
$$
\left|\sum_{1\leq\ell\leq r,1\leq k\leq n}
s_{\ell,k}\,\overline{t_{\ell,k}}\,
\right|^2=\left|\sum_{k=1}^n\left<\vec s_k,\,\vec t_k\right>_{{\mathbb C}^r}\right|^2,
$$
which by the Cauchy-Schwarz inequality is
$$
\begin{aligned}&\leq n\sum_{k=1}^n\left|\left<\vec s_k,\,\vec t_k\right>_{{\mathbb C}^r}\right|^2
\leq n\sum_{k=1}^n\left\|\vec t_k\right\|^2_{{\mathbb C}^r}\cr
&=
n\sum_{1\leq m,\ell\leq r}
\left|\sum_{k=1}^n
s_{m,k}\overline{t_{\ell,k}}
\right|^2.\cr
\end{aligned}
$$
\end{proof}

\begin{remark}\label{remark:2.3} Elements of ${\mathbb C}^r\otimes{\mathbb C}^n$ can naturally be identified with ${\mathbb C}$-linear operators between the vector spaces ${\mathbb C}^r$ and ${\mathbb C}^n$.  In the literature there are generalizations of the usual Cauchy-Schwarz inequality to bounded ${\mathbb C}$-linear operators between Hilbert spaces.  For example, Lemma 2.4 on p.193 of \cite{Haagerup1985} gives the following generalization of the usual Cauchy-Schwarz inequality
$$
\left\|\sum_{j=1}^n a_j\otimes\overline{b_j}\,\right\|_{H\otimes\overline{K}}\leq
\ \left\|\sum_{j=1}^n a_j\otimes\overline{a_j}\,\right\|_{H\otimes\overline{H}}\ \
\left\|\sum_{j=1}^n b_j\otimes\overline{b_j}\,\right\|_{K\otimes\overline{K}},
$$
where $a_1,\cdots,a_n$ (respectively $b_1,\cdots,b_n$) are bounded ${\mathbb C}$-linear operators on the Hilbert space $H$ (respectively $K$) over ${\mathbb C}$ (with appropriate interpretation for taking complex-conjugates).   The norm
$$
\left\|\sum_{j=1}^n a_j\otimes\overline{b_j}\,\right\|_{H\otimes\overline{K}}
$$
means the norm of the ${\mathbb C}$-linear operator $\sum_{j=1}^n a_j\otimes\overline{b_j}$ on the Hilbert space $H\otimes\overline{K}$, which is the supremum of the norm of the image of any element of unit norm in $H\otimes\overline{K}$.
However, the nature of such generalizations is different and the essential factor $\min(r,n)$ in the inequality in Proposition \ref{Proposition:2.1CauchySchwarzInequalityForTensors}
is rather delicate and does not occur readily in such generalizations without repeating the kind of arguments used in the proof of Proposition \ref{Proposition:2.1CauchySchwarzInequalityForTensors}.
\end{remark}

\section{Nakano Curvature of Bundle and Twisting by Metric of Trivial Line Bundle}\label{section3:NakanoCurvatureOfBundleAndTwistingByMetricOfTrivialLineBundle}

We now straightforwardly compute the Nakano curvature in the case of a special subbundle of rank $p-1$ in the trivial vector
bundle of rank $p$ whose induced metric is twisted by a special weight function ({\it i.e.},  a metric of the trivial line bundle) to conclude the
nonnegativity of the Nakano curvature.

\subsection{\it Nakano Curvature of Vector Bundles and Subbundles}\label{subsection:3.1}  Let $U$ be a domain in ${\mathbb C}^n$ with coordinates $z^1,\cdots,z^n$.  Let $V$ be a holomorphic vector bundle of rank $r$ on $U$ with smooth Hermitian metric $H$ whose components are $H_{j\bar k}$ (for $1\leq j,k\leq r$).  We use the notation $\Theta^V_{j\bar k\lambda\bar\nu}$ (for $1\leq j,k\leq r$ and $1\leq\lambda,\nu\leq n$) to denote the components of the curvature tensor $\Theta^V$ of $V$ which is defined as
$$\Theta^V_{j\bar k\lambda\bar\nu}=-\partial_\lambda\partial_{\bar\nu} H_{j\bar k}+\sum_{1\leq i,\ell \leq r}H^{\bar\ell i}\left(\partial_\lambda H_{k\bar\ell}\right)\overline{\left(\partial_\nu H_{i\bar j}\right)}.
$$
When there is no possible confusion, we will simply use $\Theta_{j\bar k\lambda\bar\nu}$ to denote $\Theta^V_{j\bar k\lambda\bar\nu}$ and use $\Theta$ to denote $\Theta^V$.

The sign convention is chosen here so that when $r=1$, the nonnegativity of the curvature of the metric $e^{-\varphi}$ means the plurisubharmonicity of the function $\varphi$.

When we have a holomorphic subbundle $W$ of $V$ with rank $1\leq s<r$ and local frame $e_1,\cdots,e_s$ so that $e_1,\cdots,e_r$ is orthonormal at the point $P$ under consideration (after the application of an $P$-dependent element of $GL(r,{\mathbb C})$), the curvature $\Theta^W_{j\bar k\lambda\bar\nu}$ for $W$ (for $1\leq j,k\leq s$ and $1\leq\lambda,\nu\leq n$) is given by
\begin{equation}\label{eq:3.1.1NakanoCurvatureOfVectorBundlesAndSubbundles1}
\begin{aligned}\Theta^W_{j\bar k\lambda\bar\nu}
&=-\partial_\lambda\partial_{\bar\nu} H_{j\bar k}+\sum_{\ell=1}^s\left(\partial_\lambda H_{j\bar\ell}\right)\overline{\left(\partial_\nu H_{\ell\bar k}\right)}\cr
&=\Theta^V_{j\bar k\lambda\bar\nu}-\sum_{\ell=s+1}^r\left(\partial_\lambda H_{j\bar\ell}\right)\overline{\left(\partial_\nu H_{\ell\bar k}\right)}\cr
\end{aligned}
\end{equation}
when $W$ is given the metric induced from the metric of $V$.

By the Nakano curvature of $V$ we mean the Hermitian form on $V\otimes T_U$ defined by
$$
\left(v^{j\lambda}\right)_{1\leq j\leq r, 1\leq\lambda\leq n}
\mapsto\sum_{1\leq j,k\leq r,1\leq\lambda,\nu\leq n}\Theta^V_{j\bar k\lambda\bar\nu}v^{j\lambda}\overline{v^{k\nu}}
$$
for $\left(v^{j\lambda}\right)_{1\leq j\leq r, 1\leq\lambda\leq n}$ in $V\otimes T_U$.  This Hermitian form was introduced by S. Nakano in (2.10) on p.8 of \cite{Nakano1955}.  When there is no possible confusion, we will simply use $\Theta^V$ or $\Theta^V_{j\bar k\lambda\bar\nu}$ to mean this Hermitian form on $V\otimes T_U$.  We also use $\Theta^V$ to mean the self-adjoint operator on the vector bundle $V\otimes T_U$ and use $\left(\Theta^V\right)^{-1}$ to mean the self-adjoint operator on $V\otimes T_U$ which is the inverse of the self-adjoint operator $\Theta^V$ when $\Theta^V$ is strictly positive.  When the self-adjoint operator $\Theta^V$ is only semipositive, we use the notation $\left(\Theta^V\right)^{-1}$ to mean the limit of $\left(\Theta^V+\varepsilon I\right)^{-1}$ as the positive number $\varepsilon$ approaches $0$, where $I$ is the identity operator of $V\otimes T_U$.  We will use this limit $\left(\Theta^V\right)^{-1}$ in the context of the inner product $\left(\left(\Theta^V\right)^{-1}F,\,F\right)$ for an element $F$ of $V\otimes T_U$ which, defined as $\lim_{\varepsilon\to 0+}\left(\left(\Theta^V+\varepsilon I\right)^{-1}F,\,F\right)$, is allowed to assume the value $+\infty$.

For another element $F^\prime$ of $V\otimes T_U$ (over the same point of $U$), by letting $\varepsilon\to 0+$ in the Cauchy-Schwarz inequality
$$
\begin{aligned}\left|\left(F,\,F^\prime\right)\right|^2&=\left|\left(\left(\Theta^V+\varepsilon I\right)^{\frac{1}{2}}\left(\Theta^V+\varepsilon I\right)^{-\frac{1}{2}}F,\,F^\prime\right)\right|^2\cr
&=\left|\left(\left(\Theta^V+\varepsilon I\right)^{-\frac{1}{2}}F,\,\left(\Theta^V+\varepsilon I\right)^{\frac{1}{2}}F^\prime\right)\right|^2\cr
&\leq\left\|\left(\Theta^V+\varepsilon I\right)^{-\frac{1}{2}}F\right\|^2\,\left\|\left(\Theta^V+\varepsilon I\right)^{\frac{1}{2}}F^\prime\right\|^2\cr
&=\left(\left(\Theta^V+\varepsilon I\right)^{-1}F,\,F\right)\,\left(\left(\Theta^V+\varepsilon I\right)F^\prime,\,F^\prime\right),\cr
\end{aligned}
$$
we conclude that
\begin{equation}\label{eq:3.1.2NakanoCurvatureOfVectorBundlesAndSubbundles2}
\left|\left(F,\,F^\prime\right)\right|^2\leq\left(\left(\Theta^V\right)^{-1}F,\,F\right)\,\left(\left(\Theta^V\right)F^\prime,\,F^\prime\right).
\end{equation}
This inequality (\ref{eq:3.1.2NakanoCurvatureOfVectorBundlesAndSubbundles2}) of Cauchy-Schwarz type also clearly holds for any semipositive tensors of the same tensor type as a Nakano curvature tensor.

Note that for a $V$-valued $(0,1)$-form $F$ at a point of $U$, in the computation of
$\left(\Theta^V\right)^{-1}F$ and $\left(\Theta^V\right)F$, the $V$-valued $(0,1)$-form $F$ is naturally identified with an element of $V\otimes T_U$ at that point with the raising of the barred subscript of $F$ to an unbarred superscript.  The same kind of manipulation involving Nakano curvature and inner product is used when $F$ is a $V$-valued $(n,1)$-form instead of a $V$-valued $(0,1)$-form.

When there is another tensor $\Xi$ with components $\Xi_{j\bar k\lambda\bar\nu}$ which defines a self-adjoint operator on $V\otimes T_U$, we say that $\Xi$ is is dominated by $\Theta^V$ or $\Xi\leq\Theta^V$ if the Hermitian form on $V\otimes T_U$ defined by $\Xi$ is $\leq$ the Hermitian form on $V\otimes T_U$ defined by $\Theta^V$.

For a $(1,1)$-form $\omega=\sum_{\lambda,\nu=1}^n\omega_{\lambda\bar\nu}dz^\lambda\wedge d\overline{z^\nu}$ we use the notation $H\otimes\omega$ to denote the tensor whose components are $H_{j\bar k}\omega_{\lambda\bar\nu}$.

\begin{proposition}[Nakano Curvature of Kernel Subbundle in Normal Fiber Coordinate]\label{Proposition:3.2NakanoCurvatureOfKernelSubbundleInNormalFiberCoordinate}  Let $g_1,\cdots,g_p$ be holomorphic functions on a domain $\Omega$ in ${\mathbb C}^n$ and let ${\mathcal K}$ be the kernel of the bundle-homomorphism $(g_1,\cdots,g_p):\left({\mathcal O}_\Omega\right)^{\oplus p}\to {\mathcal O}_\Omega$ which is given the metric induced from the standard flat metric of the trivial ${\mathbb C}$-vector bundle $\left({\mathcal O}_\Omega\right)^{\oplus p}$ of rank $p$. Let $P$ be a point of $\Omega$ with $g_1(P)=\cdots=g_{p-1}(P)=0$ and $g_p(P)=1$.  Then at the point $P$ the Nakano curvature $\Theta_{j\bar k\lambda\bar\nu}$ of the kernel vector subbundle ${\mathcal K}$ of rank $p-1$ is given by
$$
\begin{aligned}
\left(v^{j\lambda}\right)_{1\leq j\leq p-1, 1\leq\lambda\leq n}
&\mapsto\sum_{1\leq j,k\leq p-1,1\leq\lambda,\nu\leq n}\Theta_{j\bar k\lambda\bar\nu}v^{j\lambda}\overline{v^{k\nu}}
\cr
&=-\left|\sum_{1\leq j\leq p-1,1\leq\lambda\leq n}\left(\partial_\lambda g_j\right)v^{j\lambda}\right|^2\cr
\end{aligned}
$$
for $\left(v^{j\lambda}\right)_{1\leq j\leq p-1, 1\leq\lambda\leq n}$.
\end{proposition}

\begin{proof} At points of $\Omega$ where $g_1,\cdots,g_p$ are not all zero, the kernel subbundle ${\mathcal K}$ is spanned by the local basis of $p-1$ local holomorphic sections
$$
e_j=\left(0,\cdots,0,\underbrace{g_p}_{j\,{\rm th}},0,\cdots,0,-g_j\right)
$$
for $1\leq j\leq p-1$ so that the metric induced on ${\mathcal K}$ by the trivial ${\mathbb C}$-vector bundle $\left({\mathcal O}_\Omega\right)^{\oplus p}$ is given by
\begin{equation}\label{eq:3.3.1NakanoCurvatureOfKernelSubbundleInNormalFiberCoordinate}
H_{j\bar k}=\left<e_j,\,e_k\right>=\delta_{jk}\left|g_p\right|^2+g_j\overline{g_k}
\end{equation}
for $1\leq j,k\leq p-1$.  Let $e_p=(g_1,\cdots,g_p)$.  Since $g_1(P)=\cdots=g_{p-1}(P)=0$ and $g_p(P) =1$, it follows from (\ref{eq:3.3.1NakanoCurvatureOfKernelSubbundleInNormalFiberCoordinate}) that the frame $e_1,\cdots,e_p$ of $\left({\mathcal O}_\Omega\right)^{\oplus p}$ is orthonormal at $P$.  We also denote the inner product $\left<e_j,\,e_k\right>$ by $H_{j\bar k}$ when either $j$ or $k$ is $p$.

We now compute the curvature $\Theta_{j\bar k\lambda\bar\nu}$ of $H_{j\bar k}$ at the point $P$ by using the formula (\ref{eq:3.1.1NakanoCurvatureOfVectorBundlesAndSubbundles1}) with $V=\left({\mathcal O}_\Omega\right)^{\oplus p}$ and $W={\mathcal K}$ as follows.  Since $V$ is the trivial ${\mathbb C}$-vector bundle of rank $p$, we have
$\Theta^V_{j\bar k\alpha\bar\beta }\equiv 0$ and $r=p$ and $s=p-1$ so that
$$
\Theta^W_{j\bar k\alpha\bar\beta}=-\left(\partial_\alpha H_{j\bar p}\right)\overline{\left(\partial_\beta H_{k\bar p}\right)}.
$$
Since
$$
H_{j\bar p}=\left<e_j,\,e_p\right>=g_j\overline{g_p}
$$
for $1\leq j,k\leq p-1$ by (\ref{eq:3.3.1NakanoCurvatureOfKernelSubbundleInNormalFiberCoordinate}), it follows that at $P$,
$$
\Theta^W_{j\bar k\lambda\bar\nu}=-\left(\partial_\lambda H_{j\bar p}\right)\overline{\left(\partial_\nu H_{k\bar p}\right)}
=-\left(\partial_\lambda g_j\right)\overline{\left(\partial_\nu g_k\right)}
$$
for $1\leq j,k\leq p-1$
and for $$\left(v^{j\lambda}\right)_{1\leq j\leq p-1,1\leq\lambda\leq n}$$ at $P$
we have
$$
\begin{aligned}\sum_{1\leq j,k\leq p-1,1\leq\lambda,\nu\leq n}\Theta^W_{j\bar k\lambda\bar\nu}v^{j\lambda}\overline{v^{k\nu}}
&=-\sum_{1\leq j,k\leq p-1,1\leq\lambda,\nu\leq n}\left(\partial_\lambda g_j\right)\overline{\left(\partial_\nu g_k\right)}v^{j\lambda}\overline{v^{k\nu}}\cr
&=-\left|\sum_{1\leq j\leq p-1,1\leq\lambda\leq n}\left(\partial_\lambda g_j\right)v^{j\lambda}\right|^2.\cr
\end{aligned}
$$
\end{proof}

\subsection{\it Curvature Contribution from Twisting by Line Bundle}\label{subsection:3.2}  We twist the induced metric $H$ of the kernel vector bundle ${\mathcal K}$ by $\frac{1}{|g|^{2\gamma}}$ for some $\gamma>0$, where $|g|^2=\sum_{j=1}^p|g_j|^2$.  At the point $P$ of $\Omega$ where $g_1(P)=\cdots=g_{p-1}(P)=0$, the curvature contribution from the metric $\frac{1}{|g|^2}$ of the trivial line bundle is
$$
\partial\bar\partial\log|g|^2=
\frac{1}{|g|^4}\sum_{1\leq j<k\leq p}\left((\partial  g_j)g_k-
(\partial  g_k)g_j\right)\wedge\overline{\left((\partial  g_j)g_k-
(\partial  g_k)g_j\right)}
$$
which at $P$ is reduced to
\begin{equation}\label{eq:3.4.1CurvatureContributionFromTwistingByLineBundle1}
\partial\partial\log|g|^2=
\frac{1}{|g_p|^2}\sum_{1\leq j\leq p-1}(\partial g_j)\wedge\overline{(\partial g_j)}
\end{equation}
(because $(\partial g_j)g_k-
(\partial g_k)g_j$ can only be nonzero when $k=p$ and $j\not=k$ on account of $j<k$).  In addition, when $g_p(P)$ is assumed to be $1$,
\begin{equation}\label{eq:3.4.2CurvatureContributionFromTwistingByLineBundle2}
\partial\bar\partial\log|g|^2=
\sum_{1\leq j\leq p-1}(\partial g_j)\wedge\overline{(\partial g_j)}
\end{equation}
at the point $P$.
Thus, combined with the computation in Proposition \ref{Proposition:3.2NakanoCurvatureOfKernelSubbundleInNormalFiberCoordinate}, under the assumption of $g_1(P)=\cdots=g_p(P)=0$ and $g_p(P)=1$ the value at $\left(v^{j\lambda}\right)_{1\leq j\leq p-1, 1\leq\lambda\leq n}$ of the Hermitian form at $P$ defined by the Nakano curvature of the metric
$\frac{H_{j\bar k}}{|g|^{2\gamma}}$ for $\gamma>0$ is
$$
\gamma\sum_{1\leq j,\ell\leq p-1}\left|\sum_{1\leq\lambda\leq n}(\partial_\lambda g_j)v^{\ell\lambda}\right|^2
-\left|\sum_{1\leq j\leq p-1,1\leq\lambda\leq n}\left(\partial_\lambda g_j\right)v^{j\lambda}\right|^2.
$$
By applying the Cauchy-Schwarz inequality for tensors as formulated in Remark \ref{remark:2.2}
$$
\left|\sum_{1\leq\ell\leq r,1\leq k\leq n}
s_{\ell,k}t_{\ell,k}
\right|^2\leq\min(r,n)\sum_{1\leq m,\ell\leq r}
\left|\sum_{k=1}^n
s_{m,k}t_{\ell,k}
\right|^2
$$
with $r=p-1$ and $s_{\ell,\lambda}=\partial_\lambda g_\ell$ and $t_{\ell,\lambda}=v^{\ell\lambda}$, we conclude that
$$
\gamma\sum_{1\leq j,\ell\leq p-1}\left|\sum_{1\leq\lambda\leq n}(\partial_\lambda g_j)v^{\ell\lambda}\right|^2
\geq\left|\sum_{1\leq j\leq p-1,1\leq\lambda\leq n}\left(\partial_\lambda g_j\right)v^{j\lambda}\right|^2
$$
for $\gamma\geq\min(n,p-1)$ at a point $P$ with $g_1(P)=\cdots=g_p(P)=0$ and $g_p(P)=1$.

We now drop the assumption of $g_1(P)=\cdots=g_p(P)=0$ and $g_p(P)=1$ so that when an element $v$ of ${\mathcal K}$
is regarded as an element $v$ of $\left({\mathcal O}_\Omega\right)^{\oplus p}$ with components $(v_1,\cdots,v_p)$, we can apply the above computation to the general case with $v_p$ not necessarily $0$.  The following conclusion now holds.

\begin{proposition}\label{proposition:3.4.3} The Hermitian form
$$
\left(v^{j\lambda}\right)_{1\leq j\leq p, 1\leq\lambda\leq n}
\mapsto\sum_{1\leq j,k\leq p,1\leq\lambda,\nu\leq n}\Theta(\gamma)_{j\bar k\lambda\bar\nu}v^{j\lambda}\overline{v^{k\nu}}
$$
on $V\otimes T_\Omega$ defined by the curvature
$\Theta(\gamma)_{j\bar k\lambda\bar\nu}$
of the metric
$$H(\gamma):=\frac{H_{j\bar k}}{|g|^{2\gamma}}$$ for $\gamma\geq\min(n,p-1)$ dominates $\gamma-\min(n,p-1)$ times the Hermitian form
$$
\left(v^{j\lambda}\right)_{1\leq j\leq p, 1\leq\lambda\leq n}
\mapsto\left(\sum_{1\leq\lambda,\nu\leq n}\partial_\lambda\partial_{\bar\nu}\log|g|^2\right)\left(\sum_{j=1}^p
v^{j\lambda}\overline{v^{j\nu}}\right)
$$
defined by (the tensor product of the identity operator of $V$ and) $\partial\bar\partial\log|g|^2$ on $V\otimes T_\Omega$.
\end{proposition}

\section{Vanishing Theorem and Solution of $\bar\partial$ Equation for Nonnegative Nakano Curvature}\label{section4:VanishingTheoremAndSolutionOfDBarEquationForNonnegativeNakanoCurvature}

We now present the vanishing theorem and solution of the $\bar\partial$-equation for a holomorphic vector bundle of nonnegative Nakano curvature on a strictly pseudoconvex manifold, with emphasis on the lack of strict positivity of the curvature at any point.  Even for application to algebraic geometry, the strictly pseudoconvex situation is needed to enable the removal of an ample complex hypersurface to handle the singularity of the metrics of the vector bundle which is inherent to our problem at hand.

Though the setting of Skoda's ideal generation is for a Stein domain spread over ${\mathbb C}^n$ whose canonical line bundle is trivial, in this section we formulate the vanishing theorem in the setting of a general complex manifold $X$ of complex dimension $n$ whose canonical line bundle $K_X$ may not be trivial, where the solution of the $\bar\partial$-equation has to be formulated with the right-hand side being a vector-bundle-valued $(n,1)$-form and the unknown being a vector-bundle-valued $(n,0)$-form.  Later in the application of the vanishing theorem to Skoda's ideal generation, the vector-bundle-valued $(n,1)$-form on the right-hand side and the vector-bundle-valued $(n,0)$-form as the unknown will be changed respectively to the vector-bundle-valued $(0,1)$-form on the right-hand side and the vector-bundle-valued function as the unknown.

\begin{theorem}[Vanishing Theorem of Kodaira-Nakano with $L^2$ Estimate for Pseudoconvex Manifolds]\label{theorem:4.1VanishingTheoremKodairaNakanoL2EstimatePseudoconvexDomain}  Let $\Omega$ be a relatively compact domain in an $n$-dimensional complex manifold $X$ with K\"ahler metric $g_{\lambda\bar\nu}$ (for $1\leq\lambda,\nu\leq n$) such that the boundary of $\Omega$ is smooth and strictly psuedoconvex.  Let $V$ be a holomorphic vector bundle of rank $r$ on $X$ with smooth metric $H_{j\bar k}$ (for $1\leq j,k\leq r$) whose Nakano curvature $\Theta_{j\bar k\lambda\bar\nu}$ is semipositive on $X$.  Let $F$ be a smooth $V$-valued $\bar\partial$-closed $(n,1)$-form on $\Omega$ such that $\left(\Theta^{-1}F,\,F\right)_{L^2(\Omega)}$ is finite (in the sense of {\bf\ref{subsection:3.1}}). Then the equation $\bar\partial u=F$ can be solved for a smooth $V$-valued $(n,0)$-form over $\Omega$ with $\left\|u\right\|^2_{L^2(\Omega)}\leq\left(\Theta^{-1}F,\,F\right)_{L^2(\Omega)}$.
\end{theorem}

\subsection{}\label{subsection:4.2} Though this statement may not be found in the literature as stated, it can be routinely derived from known available techniques.  Instead of giving a detailed proof of it here, we will only comment on the noteworthy features of the statement and explain the main lines of arguments for its proof.  The most important feature of the statement is that there may be no point in $X$ where the Nakano curvature $\Theta$ of the holomorphic vector bundle $V$ is strictly positive.  At a point $P$ of $X$ the Hermitian form of the Nakano curvature $\Theta$ may be positive when evaluated at some elements $V\otimes T_X$ at $P$ and may be zero at some other elements of $V\otimes T_X$ at $P$.  This may be the situation at every point $P$ of $X$.  The assumption of the finiteness of $\left(\Theta^{-1}F,\,F\right)_{L^2(\Omega)}$ means the vanishing of components of $F$ which correspond to elements of $\left.\left(V\otimes T_X\right)\right|_\Omega$ where the Hermitian form of the Nakano curvature $\Theta$ vanishes.

The key argument of solving the equation $\bar\partial u=F$ on $\Omega$ for a smooth $V$-valued $(n,0)$-form $u$ on $\Omega$ is to use the following basic estimate obtained by completion of squares by integration by parts
$$
\left\|\bar\partial v\right\|_{L^2(\Omega)}^2+\left\|\bar\partial^*v\right\|_{L^2(\Omega)}^2
=\left\|\bar\nabla v\right\|_{L^2(\Omega)}^2+\left({\rm Levi}_{{}_{\partial\Omega}},\,v\wedge\bar v\right)_{L^2(\partial\Omega)}+\left(\Theta v,v\right)_{L^2(\Omega)}
$$
for a test $V$-valued $(n,1)$-form $v$ on $\Omega$ which belongs to the domain of $\bar\partial$ and $\bar\partial^*$ on $\Omega$, where ${\rm Levi}_{{}_{\partial\Omega}}$ is the Levi form $\partial\bar\partial\rho$ of $\partial\Omega$ (with $\rho$ being a smooth function on $\partial\Omega$ defining $\Omega$ as $\Omega=\{r<0\}$ and $d\rho$ identically $1$ on $\partial\Omega$) and $\bar\nabla$ means the covariant differential operator along tangent vectors of type $(0,1)$.
If one has the estimate
\begin{equation}\label{eq:4.2.1}
\left|(v,\,F)_{L^2(\Omega)}\right|\leq \hat C\left\|\bar\partial^*v\right\|_{L^2(\Omega)}
\end{equation}
for all test $V$-valued $(n,1)$-form $v$ on $\Omega$ in the domains of the two operators $\bar\partial$ and $\bar\partial^*$, one can apply Riesz's representation theorem to the ${\mathbb C}$-linear functional
\begin{equation}\label{eq:4.2.2}
\bar\partial^*v\mapsto(v,F)_{L^2(\Omega)}
\end{equation}
with bound $\hat C$ to write the ${\mathbb C}$-linear functional (\ref{eq:4.2.2}) in the form $w\mapsto(w,u)_{L^2(\Omega)}$ for some $u$ with $\left\|u\right\|_{L^2(\Omega)}\leq\hat C$ so that
from
$$
(v,\,F)_{L^2(\Omega)}=(\bar\partial^*v,\,u)_{L^2(\Omega)}=(v,\bar\partial u)_{L^2(\Omega)}
$$
for all test $V$-valued $(n,1)$-form $v$ on $\Omega$ in the domains of $\bar\partial$ and $\bar\partial^*$
it follows that $\bar\partial u=F$ with $\left\|u\right\|_{L^2(\Omega)}\leq \hat C$.
To obtain the estimate (\ref{eq:4.2.1}), since $F$ is $\bar\partial$-closed, it suffices to get the estimate for test $V$-valued $(n,1)$-forms $v$ which are $\bar\partial$-closed.  By the inequality (\ref{eq:3.1.2NakanoCurvatureOfVectorBundlesAndSubbundles2}) of Cauchy-Schwarz type we have
$$
\begin{aligned}
\left|(v,\,F)_{L^2(\Omega)}\right|^2&\leq\left|\left(\Theta^{-1}F,\,F\right)_{L^2(\Omega)}\right|\cdot\left|(\Theta v,\,v)_{L^2(\Omega)}\right|\cr
&\leq\left|\left(\Theta^{-1}F,\,F\right)_{L^2(\Omega)}\right|\left(\left\|\bar\partial v\right\|_{L^2(\Omega)}^2+\left\|\bar\partial^* v\right\|_{L^2(\Omega)}^2\right)\cr
&=\left|\left(\Theta^{-1}F,\,F\right)_{L^2(\Omega)}\right|\cdot\left\|\bar\partial^* v\right\|^2_{L^2(\Omega)}\cr
\end{aligned}
$$
so that the estimate (\ref{eq:4.2.1}) holds with $\hat C=\left|\left(\Theta^{-1}F,\,F\right)_{L^2(\Omega)}\right|^{\frac{1}{2}}$, which according to the assumption is finite.  This finishes the argument, because $\hat C$ is equal to the square root of $C=\left|\left(\Theta^{-1}F,\,F\right)_{L^2(\Omega)}\right|$ and $\left\|u\right\|^2_{L^2(\Omega)}\leq C$.

Here we have left out the routine details about the use of convolution to handle the problem of proving that smooth $V$-valued $(n,1)$-forms on $\Omega$ up to the boundary in domain of $\bar\partial^*$ are dense, with respect to the graph norm, in the space of all $L^2$ $V$-valued $(n,1)$-forms on $\Omega$ in the domains of  $\bar\partial$ and $\bar\partial^*$.

More general statements than that given in Theorem \ref{theorem:4.1VanishingTheoremKodairaNakanoL2EstimatePseudoconvexDomain} hold by the same arguments. For example, the $V$-valued $(n,1)$-form $F$ can be replaced by a $V$-valued $(n,q)$-form with $q\geq 1$ and the condition of $\partial\Omega$ being smooth and strictly pseudoconvex can be weakened.  Here we only give the statement which suffices for our purpose.

\begin{corollary}[Special Case of Semipositive $(1,1)$-Form as Lower Bound]\label{corollary:4.3} In Theorem \ref{theorem:4.1VanishingTheoremKodairaNakanoL2EstimatePseudoconvexDomain} suppose there are a smooth semipositive $(1,1)$-form
$\omega$ on $X$ and a positive $(n,n)$-form $\Phi$ on $X$ such that the Nakano curvature $\Theta$ of $V$ dominates $H\otimes\omega$ on $X$ and $\Phi\omega$ dominates (the nonnegative quadratic form defined by the coefficient of) the fiber trace
$${\rm FbTr}(F)=\sum_{\lambda,\nu=1}^n\left(\sum_{j,k=1}^r H_{j\bar k}F^j_{\bar\nu}\overline{F^k_{\bar\lambda}}\right)(dz^\lambda\wedge d\overline{z^\nu})
\left|dz^1\wedge\cdots\wedge dz^n\right|^2
$$
of the $\bar\partial$-closed $V$-valued $(n,1)$-form $F$ on $\Omega$,
where the components of $F$ with respect to the local fiber coordinates of $V$ are $F^1,\cdots,F^r$ and $$F^j=\sum_{\lambda=1}^n F^j_{\bar\lambda}(dz^1\wedge\cdots\wedge dz^n)\wedge d\overline{z^\lambda}.$$
Let $n_\omega$ be a positive integer $\leq n$ such that the number of nonzero eigenvalues of $\omega$ is $\leq n_\omega$ at every point of $\Omega$. If the integral
$\int_\Omega\Phi$ is finite, then the $\bar\partial$-equation $\bar\partial u=F$ on $\Omega$ can be solved for the unknown smooth $V$-valued $(n,0)$-form $u$ over $\Omega$ such that the square of the $L^2$ norm of $u$ on $\Omega$ with respect to $H$ is $\leq n_\omega\int_\Omega\Phi$.
\end{corollary}

\begin{proof}  At any point $P$ of $\Omega$ we can choose local coordinates $(z^1,\cdots,z^n)$ centered at $P$ such that both $\omega$ and $H$ are diagonalized at $P$ and
$$\Phi\,\omega=\left(\sum_{\lambda=1}^{n_\omega} a_\lambda dz^\lambda\wedge d\overline{z^{\lambda}}\right)\left|dz^1\wedge\cdots\wedge dz^n\right|^2
$$
and
$${\rm FrTr}(F)=\left(\sum_{\lambda=1}^{n_\omega} b_\lambda dz^\lambda\wedge d\overline{z^{\lambda}}\right)\left|dz^1\wedge\cdots\wedge dz^n\right|^2$$
at $P$ with $0\leq b_\lambda\leq a_\lambda$ for $1\leq\lambda\leq n_\omega$.  Since the pointwise inner product
$\left<(H\otimes\omega)^{-1}F,\,F\right>_P$ at the point $P$ is equal to
$$
\sum_{\lambda=1}^{n_\omega}\frac{\Phi}{a_\lambda}\sum_{j,k=1}^r H^{\bar k j}F_{\bar k\bar\lambda}\overline{F_{\bar j\bar\lambda}}=\sum_{\lambda=1}^{n_\omega}\frac{b_\lambda\Phi}{a_\lambda}\leq n_\omega\Phi,
$$
it follows that
$$
\begin{aligned}\left(\Theta^{-1}F,\,F\right)_{L^2(\Omega)}&\leq
\left((H\otimes\omega)^{-1}F,\,F\right)_{L^2(\Omega)}\cr
&=\int_{P\in\Omega}\left<(H\otimes\omega)^{-1}F,\,F\right>_P\leq n_\omega\int_\Omega\Phi\cr
\end{aligned}
$$
and the conclusion follows from Theorem \ref{theorem:4.1VanishingTheoremKodairaNakanoL2EstimatePseudoconvexDomain}.
\end{proof}

\begin{corollary}[Solution Estimate from Base Trace of Fiber Trace of Kernel Bundle Valued $(n,1)$-Form]\label{corollary:4.5} In Theorem \ref{theorem:4.1VanishingTheoremKodairaNakanoL2EstimatePseudoconvexDomain} suppose there are a smooth semipositive $(1,1)$-form
$\omega$ on $X$ and a positive smooth $(n,n)$-form $\Phi$ on $X$ such that the Nakano curvature $\Theta$ of $V$ dominates $H\otimes\omega$ on $X$ and at every point of $\Omega$ the null space of $\omega$ is contained in the null space
of the fiber trace ${\rm FbTr}(F)$
of the $\bar\partial$-closed $V$-valued $(n,1)$-form $F$ on $\Omega$.  Let ${\rm Tr}_\omega\left({\rm FbTr}(F)\right)$ be the trace of ${\rm FbTr}(F)$ with respect to $\omega$ (which means the limit, as $\varepsilon\to 0+$, of the trace of ${\rm FbTr}(F)$ with respect to the sum of $\omega$ and $\varepsilon$ times the K\"ahler form of $X$). If the integral
$\int_\Omega{\rm Tr}_\omega\left({\rm FbTr}(F)\right)$ is finite, then
the $\bar\partial$-equation $\bar\partial u=F$ on $\Omega$ can be solved for the unknown smooth $(n,0)$-form $u$ of $V$ over $\Omega$ such that the square of the $L^2$ norm of $u$ on $\Omega$ with respect to $H$ is $\leq\int_\Omega{\rm Tr}_\omega\left({\rm FbTr}(F)\right)$.
\end{corollary}

\begin{proof} It follows from Theorem \ref{theorem:4.1VanishingTheoremKodairaNakanoL2EstimatePseudoconvexDomain} and
$$
\left(\Theta^{-1}F,\,F\right)_{L^2(\Omega)}\leq\left(\left(H\otimes\omega\right)^{-1}F,\,F\right)_{L^2(\Omega)}=\int_\Omega{\rm Tr}_\omega\left({\rm FbTr}(F)\right).
$$
\end{proof}

\section{Proof of Skoda's Result on Ideal Generation}\label{section5:ProofOfSkodasResultOnIdealGeneration}

We now present our proof of Skoda's result on ideal generation by using the above three ingredients.

\subsection{\it Setup and Solution of $\bar\partial$ Equation}\label{subsection:5.1}  Let $\Omega$ be a Stein domain spread over ${\mathbb C}^n$ and let $g_1,\cdots,g_p$ be holomorphic functions on $\Omega$ and ${\mathcal K}$ be the kernel subbundle of the bundle-homomorphism $(g_1,\cdots,g_p):\left({\mathcal O}_\Omega\right)^{\oplus p}\to {\mathcal O}_\Omega$.  Let $\psi$ be a plurisubharmonic function on $\Omega$.  Let $f$ be a holomorphic function on $\Omega$.  Let $Z$ be a complex hypersurface in $\Omega$ which contains the common zero-set of $g_1,\cdots,g_p$ such that $\Omega-Z$ is Stein.  On $\Omega-Z$ we can write
\begin{equation}\label{eq:5.1.1}
f=\sum_{j=1}^n\frac{\overline{g_j}}{|g|^2}\,f g_j
\end{equation}
and let
\begin{equation}\label{eq:5.1.2}
F=(F_1,\cdots,F_p)=\left(\bar\partial\left(\frac{\overline{g_1}}{|g|^2}\right),\cdots,\bar\partial\left(\frac{\overline{g_p}}{|g|^2}\right)\right)f.
\end{equation}
Then $F$ is a ${\mathcal K}$-valued $(0,1)$-form on $\Omega-Z$, because $\sum_{j=1}^p F_jg_j\equiv 0$ follows by applying $\bar\partial$ to (\ref{eq:5.1.1}).
We consider solving $\bar\partial u=F$ for $u$ in the kernel bundle ${\mathcal K}$ with $L^2$ estimates on the Stein domain $\Omega-Z$ spread over ${\mathbb C}^n$.  As remarked earlier, when we apply the vanishing theorem in \S\ref{section4:VanishingTheoremAndSolutionOfDBarEquationForNonnegativeNakanoCurvature}, we will simply regard the ${\mathcal K}$-valued $(0,1)$-form $F$ on $\Omega$ naturally as a ${\mathcal K}$-valued $(n,1)$-form on $\Omega$ and regard the ${\mathcal K}$-valued function $u$ on $\Omega$ naturally as a ${\mathcal K}$-valued $(n,0)$-form on $\Omega$.

We find relatively compact subdomains $\Omega_m$ of $\Omega$ with smooth strictly pseudoconvex boundary (for $m\in{\mathbb N}$) such that (i) $\Omega_m$ is relatively compact in $\Omega_{m+1}$ for $m\in{\mathbb N}$ and (ii) the union of all $\Omega_m$ for $m\in{\mathbb N}$ is equal to $\Omega-Z$.  By using smoothing by convolution we can find a smooth plurisubharmonic function $\psi_m$ on an open neighborhood of the closure of $\Omega_m$ in $\Omega_{m+1}$ such that $\psi\leq\psi_{m+1}\leq\psi_m$ on $\Omega_m$ and $\psi_m$ approaches $\psi$ on $\Omega-Z$ as $m\to\infty$.

Let $q=\min(n,p-1)$ and $\gamma>0$.
Assume that
\begin{equation}\label{eq:5.1.3}
\hat C=\int_\Omega\frac{|f|^2e^{-\psi}}{|g|^{2(q+1+\gamma)}}<\infty.
\end{equation}

We are going to apply Corollary \ref{corollary:4.3} to the kernel vector subbundle ${\mathcal K}$ on $\Omega_m$ with the metric
$$
H(\gamma,m)_{j\bar k}:=\frac{H_{j\bar k}e^{-\psi_m}}{\left|g\right|^{2(q+\gamma)}},
$$
where $H_{j\bar k}$ is the metric for ${\mathcal K}$ induced from the standard flat metric of $\left({\mathcal O}_\Omega\right)^{\oplus p}$.  Let $\omega(\gamma)=\gamma\partial\bar\partial\log|g|^2$.  By (\ref{eq:3.4.2CurvatureContributionFromTwistingByLineBundle2}) we have the bound $n_{\omega(\gamma)}\leq p-1$ for the number $n_{\omega(\gamma)}$ defined in Corollary \ref{corollary:4.3} when we do the computation of the number of nonzero eigenvalues of $\omega(\gamma)$ at a point $P$ by using the normal fiber coordinates with $g_1(P)=\cdots=g_{p-1}(P)=0$ and $g_p(P)=1$ as in {\bf{\ref{subsection:3.2}}}.  Since clearly $n_{\omega(\gamma)}$ cannot greater than $n$, it follows that $n_{\omega(\gamma)}\leq q$. By (\ref{eq:3.5.1AnotherExpressionForCurvatureContributionOfTwisting1}) we have
\begin{equation}\label{eq:5.1.4}
\Phi_m\omega(\gamma)={\rm FbTr}(F)\quad{\rm with}\quad\Phi_m=\frac{|f|^2e^{-\psi_m}}{|g|^{2(q+\gamma)}}\frac{1}{\gamma |g|^2},
\end{equation}
where the factor $\frac{|f|^2e^{-\psi_m}}{|g|^{2(q+\gamma)}}$ comes from the metric $\frac{\delta_{jk}e^{-\psi_m}}{|g|^{2(q+\gamma)}}$ of ${\mathcal O}_\Omega^{\oplus p}$ which induces the metric $H(\gamma,m)$ on the subbundle ${\mathcal K}$ of ${\mathcal O}_\Omega^{\oplus p}$.  Here, as remarked earlier, because of the triviality of the canonical line bundle of the Stein domain $\Omega$ spread over ${\mathbb C}^n$, we simply represent naturally the $(n,n)$-form $\Phi_m$ as a function.  By Corollary \ref{corollary:4.3}, (\ref{eq:5.1.3}), and (\ref{eq:5.1.4}), from $\psi\leq\psi_m$ on $\Omega_m$ it follows that there exists some smooth section $u^{(m)}$ of ${\mathcal K}$ over $\Omega_m$ such that $\bar\partial u_m=F$ on $\Omega_m$ with
$$
\begin{aligned}&\int_{\Omega_m}\frac{\sum_{j=1}^p\left|u^{(m)}_j\right|^2 e^{-\psi_m}}{|g|^{2(q+\gamma)}}\leq
n_{\omega(\gamma)}\int_\Omega\Phi_m\cr
&\leq\frac{n_{\omega(\gamma)}}{\gamma}\int_{\Omega_m}\frac{|f|^2e^{-\psi_m}}{|g|^{2(q+1+\gamma)}}=
\frac{n_{\omega(\gamma)}}{\gamma}\hat C\leq \frac{q}{\gamma}\hat C,\cr
\end{aligned}
$$
where $u^{(m)}_j$ is the $j$-component of $u^{(m)}$ when it is regarded as an element of ${\mathcal O}_\Omega^{\oplus p}$ over $\Omega_m$.  By the standard process of using a subsequence $\{m_\nu\}$ of $\{m\}$ to pass to the limit of of $u^{(m_\nu)}$ when $\nu\to\infty$, we conclude that there exists a smooth section $u$ of ${\mathcal K}$ over $\Omega-Z$ such that $\bar\partial u=F$ on $\Omega-Z$ and
$$
\int_{\Omega-Z}\frac{\sum_{j=1}^p|u_j|^2 e^{-\psi}}{|g|^{2(q+\gamma)}}\leq\frac{q}{\gamma}\hat C,
$$
where $u_j$ is the $j$-component of $u$ when it is regarded as an element of ${\mathcal O}_\Omega^{\oplus p}$ over $\Omega-Z$.  Let $h_j=\frac{\bar g_j}{|g|^2}f-u_j$.  Then $h_j$ is a holomorphic section of ${\mathcal K}$ over $\Omega-Z$ with $\sum_{j=1}^p h_jg_j=f$ on $\Omega-Z$ and
$$
\begin{aligned}&
\int_{\Omega-Z}\frac{\sum_{j=1}^p\left|h_j\right|^2 e^{-\psi}}{|g|^{2(q+\gamma)}}\cr
&=\int_{\Omega-Z}\frac{\sum_{j=1}^p\left|\frac{\bar g_j}{|g|^2}f-u_j\right|^2 e^{-\psi}}{|g|^{2(q+\gamma)}}\cr
&\leq 2\int_{\Omega-Z}\frac{\sum_{j=1}^p\left|\frac{\bar g_j}{|g|^2}f \right|^2 e^{-\psi}+\sum_{j=1}^p\left|u_j\right|^2 e^{-\psi}}{|g|^{2(q+\gamma)}}\cr
&=2\int_{\Omega-Z}\frac{\left|f\right|^2 e^{-\psi}}{{|g|^{2(q+1+\gamma)}}}+\int_\Omega\sum_{j=1}^p\frac{\left|u_j\right|^2 e^{-\psi}}{|g|^{2(q+\gamma)}}\cr
&\leq 2\left(1+\frac{q}{\gamma}\right)\,\hat C.\cr
\end{aligned}
$$
We can now extend $h_j$ to be a holomorphic section of ${\mathcal K}$ over $\Omega$ from the $L^2$ estimate.  This finishes our proof of Skoda's result on ideal generation, except for a factor of $2$ in the $L^2$ estimate for the solution $h_1,\cdots,h_p$, which we now discuss.

\subsection{}\label{subsection:5.2} Though it is an insignificant point for our purpose, there is a factor of $2$ which should not be there on the right-hand side of the estimate
$$
\int_\Omega\frac{\sum_{j=1}^p\left|h_j\right|^2 e^{-\psi}}{|g|^{2(q+\gamma)}}
\leq 2\left(1+\frac{q}{\gamma}\right)\,\hat C
$$
in {\bf{\ref{subsection:5.1}}}.
The constant $\gamma$ is related to the constant $\alpha$ in Theorem \ref{theorem:1.1} by $q+1+\gamma=\alpha q+1$ which means $\alpha=1+\frac{\gamma}{q}$ so that $\frac{\alpha}{\alpha-1}=1+\frac{q}{\gamma}$.  In its computation in {\bf{\ref{subsection:5.1}}} the estimate of
$$
\left|2{\rm Re}\left(\frac{\bar g_j}{|g|^2}f\, u_j\right)\right|
\leq\left|\frac{\bar g_j}{|g|^2}f\right|^2+\left|u_j\right|^2
$$
can be replaced by the better inequality of
$$
\left|2{\rm Re}\left(\frac{\bar g_j}{|g|^2}f\,\overline{u_j}\right)\right|
\leq\beta\left|\frac{\bar g_j}{|g|^2}f\right|^2+\frac{1}{\beta}\left|u_j\right|^2
$$
with $\beta=\sqrt{\frac{q}{\gamma}}$ to give the smaller bound of $\left(1+\sqrt{\frac{q}{\gamma}}\right)^2$ which is still greater than $1+\frac{q}{\gamma}$.

One possible way to optimize better the bound is to use the fact that since each solution section $u^{(m)}=\left(u^{(m)}_1,\cdots,u^{(m)}_p\right)$ of the bundle ${\mathcal K}$ over $\Omega_m$ is obtained from the Riesz representation theorem, the section $u^{(m)}$ is orthogonal to the kernel of the $\bar\partial$ operator on $\Omega_m$ with respect to the metric $H(\gamma,m)_{j\bar k}$ on $\Omega^{(m)}$.  In this note we will not go further into the question of optimization of the constant in the integral bound of the solution of ideal generation.

\section{Variants of Skoda's Ideal Generation}\label{section6:VariantsOfSkodasIdealGeneration}

From our proof variants of Skoda's ideal generation result can readily be obtained by different choices of the weight function (which is the twisting by
the metric of the trivial line bundle).  Here we give three examples in the following theorem.  The example in Theorem \ref{theorem:6.1}(b) yields Th\'eor\`eme 2 on p.571 in Skoda's paper \cite{Skoda1972} with the careful choice of $\varphi$ given there.

\begin{theorem}[Variants of Skoda's Ideal Generation]\label{theorem:6.1}
Let $\Omega$ be a domain spread over ${\mathbb C}^n$ which is
Stein. Let $\psi$ be a plurisubharmonic function on $\Omega$,
$g_{1},\ldots ,g_{p}$ be holomorphic functions on $\Omega$ (not all identically zero), $\gamma
>0$, $q=\min \left(n,p-1\right)$, $|g|^2=\sum_{j=1}^p\left|g_j\right|^2$, and $f$ be a holomorphic function on $\Omega$.

\begin{enumerate}[(a)]\item If the integral
$$
C_1:=\int_\Omega\frac{|f|^2\left(|g|^2+q(1+|g|^2)\right)e^{-\psi}}{|g|^{2(q+2)}(1+|g|^2)}
$$
is finite, then there exist holomorphic functions $h_1,\ldots ,h_p$ on
$\Omega$ with $f=\sum_{j=1}^p h_jg_j$ on $\Omega$ such that
$$
\int_\Omega\frac{\sum_{j=1}^p\left|h_j\right|^2 e^{-\psi}}{|g|^{2q}(1+|g|^2)}\leq 2C_1.
$$
\item Let $\varphi$ be a smooth plurisubharmonic function on $\Omega$ such that $\omega:=\partial\bar\partial\varphi$ is
not identically zero on $\Omega$.  If the integral
$$
C_2:=\int_\Omega\frac{|f|^2(1+\Delta_\omega\log|g|^2)e^{-(\varphi+\psi)}}{|g|^{2(q+1)}}
$$
is finite, then there exist holomorphic functions $h_1,\ldots ,h_p$ on
$\Omega$ with $f=\sum_{j=1}^p h_jg_j$ on $\Omega$ such that
$$
\int_\Omega\frac{\sum_{j=1}^p\left|h_j\right|^2 e^{-(\varphi+\psi)}}{|g|^{2q}}\leq 2C_2,
$$
where $\Delta_\omega\log|g|^2$ means the Laplacian of the function $\log|g|^2$ with respect to $\omega$ (which means the
limit as $\varepsilon\to 0+$ of the Laplacian of the function $\log|g|^2$ with respect to the sum of $\omega$ and $\varepsilon$ times
the standard Euclidean K\"ahler form of ${\mathbb C}^n$).
\item If $|g_j|<1$ on $\Omega$ for $1\leq j\leq p$ and
$$
C_3:=\int_\Omega\frac{|f|^2\left(\left(\log\left(\frac{1}{|g|^2}\right)\right)+q\left(\log\left(\frac{1}{|g|^2}\right)\right)^2\right)e^{-\psi}}{|g|^{2(q+1)}}
$$
is finite,
then there exist holomorphic functions $h_1,\ldots ,h_p$ on
$\Omega$ with $f=\sum_{j=1}^p h_jg_j$ on $\Omega$ such that
$$\int_\Omega\frac{\sum_{j=1}^p\left|h_j\right|^2\left(\log\left(\frac{1}{|g|^2}\right)\right) e^{-\psi}}{|g|^{2q}}
\leq 2 C_3.
$$
\end{enumerate}
\end{theorem}

\begin{proof} For the proof of Part (a), since
$$
\begin{aligned}\partial\bar\partial\log(1+|g|^2)&=\frac{|g|^2\partial\bar\partial\log|g|^2+\sum_{k=1}^p\partial g_k\wedge\overline{\partial g_k}}{(1+|g|^2)^2}
\cr
&\geq\frac{|g|^2\partial\bar\partial\log|g|^2}{1+|g|^2},\cr
\end{aligned}
$$
it follows that Corollary \ref{corollary:4.3} can be applied to the subbundle ${\mathcal K}$ of $\left({\mathcal O}_\Omega\right)^{\oplus p}$ with the metric induced by the standard flat metric of
$\left({\mathcal O}_\Omega\right)^{\oplus p}$ multiplied by the weight function
$$
\frac{|f|^2e^{-\psi}}{|g|^{2q}(1+|g|^2)}
$$
so that the function $\Phi$ can be chosen to be
$$
\frac{|f|^2 e^{-\psi}}{|g|^{2q+4}}
$$
to solve the equation $\bar\partial u=F$ for a smooth section of ${\mathcal K}$ on $\Omega$ with $F$ given by (\ref{eq:5.1.2})
as described in {\bf{\ref{subsection:5.1}}}.

The proof of Part (b) follows immediately from Corollary \ref{corollary:4.5} with the same meaning for $\omega$ as in Corollary \ref{corollary:4.5} and with the fiber trace ${\rm FbTr}(F)$ of
$F$ of (\ref{eq:5.1.2}) being equal to $\frac{1}{|g|^2}\Delta_\omega\log|g|^2$.

For the proof of Part (c), since
$$
\begin{aligned}
-\partial\bar\partial\log\log\left(\frac{1}{|g|^2}\right)&=
\frac{\partial\bar\partial\log|g|^2}{\log\left(\frac{1}{|g|^2}\right)}+\frac{\left(\partial\log|g|^2\right)\wedge\left(\bar\partial\log|g|^2\right)}{\left(\log\left(\frac{1}{|g|^2}\right)\right)^2}\cr
&\geq\frac{\partial\bar\partial\log|g|^2}{\log\left(\frac{1}{|g|^2}\right)},\cr
\end{aligned}
$$
it follows that Corollary \ref{corollary:4.3} can be applied to the subbundle ${\mathcal K}$ of $\left({\mathcal O}_\Omega\right)^{\oplus p}$ with the metric induced by the standard flat metric of
$\left({\mathcal O}_\Omega\right)^{\oplus p}$ multiplied by the weight function
$$
\frac{\left(\log\left(\frac{1}{|g|^2}\right)\right)e^{-\psi}}
{|g|^{2q}}
$$
so that the function $\Phi$ can be chosen to be
$$
\frac{1}{|g|^2}\frac{\left(\log\left(\frac{1}{|g|^2}\right)^2\right)e^{-\psi}}{|g|^{2q}}|f|^2
$$
to solve the equation $\bar\partial u=F$ for a smooth section of ${\mathcal K}$ on $\Omega$ with $F$ given by (\ref{eq:5.1.2})
as described in {\bf{\ref{subsection:5.1}}}.
\end{proof}

\begin{remark} Theorem \ref{theorem:6.1}(b) can be applied to any $\varphi=\frac{1}{(1+|z|^2)^m}$ (for $m\in{\mathbb N}$) or $\varphi=e^{-|z|^2}$ or any other global
smooth strictly plurisubharmonic function on ${\mathbb C}^n$, but the use of any weight function $e^{-\varphi}$ with $\varphi$ strictly plurisubharmonic defeats the original purpose
of applying vanishing theorems and $L^2$ estimates of $\bar\partial$ without the Nakano curvature being strictly positive at any point.  Unfortunately, unlike Theorem \ref{theorem:1.1} (with Theorem \ref{theorem:1.2} as its algebraic geometric formulation), none of the three parts of Theorem \ref{theorem:6.1} can be formulated and used in an algebraic geometric setting.
\end{remark}

\section{Comparison with Original Proof of Skoda's Ideal Generation}\label{section7:ComparisonWithOriginalProofOfSkodasIdealGeneration}

The proof given here is simpler and more straightforward than Skoda's original proof, because (i) we can quote directly well-developed known techniques and geometric notions  and (ii) we can use normal coordinates to simplify the computations on account of the coordinate-independence of the geometric entities involved.  However, behind the facade of geometric and analytic arguments, the manipulations in computation between our proof and Skoda's original proof can be put in parallel correspondence, with expressions in our proof neater and simpler than in Skoda's proof due to the use of normal coordinates.  For the discussion on the comparison, we start out with the following Cauchy-Schwarz inequality (with a special factor) for tensors more general than those in Proposition \ref{Proposition:2.1CauchySchwarzInequalityForTensors}.

\begin{proposition}[Cauchy-Schwarz Inequality for More General Tensors]\label{proposition:7.1}
Let ${\vec a}\in{\mathbb C}^p$.  Let ${\mathbf b}$ and ${\mathbf c}$ be elements of ${\mathbb C}^p\otimes{\mathbb C}^n$.  Let $q=\min(n,p-1)$.  Then
$$\left|\left<\vec a\wedge{\mathbf b},\,\vec a\wedge{\mathbf c}
\right>_{{}_{\left(\bigwedge^2{\mathbb C}^p\right)\times{\mathbb C}^n}}\right|^2
\leq
q\,\left\|\vec a\right\|^2_{{}_{{\mathbb C}^p}}
\left\|\left<\vec a\wedge{\mathbf b},\,{\mathbf c}\right>_{{}_{{\mathbb C}^n}}\right\|^2_{{}_{\left(\bigwedge^2{\mathbb C}^p\right)\times{\mathbb C}^p}}.
$$
Here (i) $\vec a\wedge{\mathbf b}$ is the wedge product of $\vec a$
separately with each of the $n$ elements of ${\mathbb C}^p$ defined by ${\mathbf b}$ to end up with an element of
$\left(\bigwedge^2{\mathbb C}^p\right)\times{\mathbb C}^n$ and (ii) the inner product
$\left<\vec a\wedge{\mathbf b},\,{\mathbf c}\right>_{{}_{{\mathbb C}^n}}$ is taken when
the element $\vec a\wedge{\mathbf b}$ of $\left(\bigwedge^2{\mathbb C}^p\right)\times{\mathbb C}^n$ is regarded as
an $\left(\bigwedge^2{\mathbb C}^p\right)$-valued element of ${\mathbb C}^n$ and the element ${\mathbf c}$ is regarded
as a ${\mathbb C}^p$-valued element of
${\mathbb C}^n$ so that $\left<\vec a\wedge{\mathbf b},\,{\mathbf c}\right>_{{}_{{\mathbb C}^n}}$ becomes an element of
$\left(\bigwedge^2{\mathbb C}^p\right)\times{\mathbb C}^p$.
\end{proposition}

\subsection{\it Equivalence of Proposition \ref{proposition:7.1} and Proposition \ref{Proposition:2.1CauchySchwarzInequalityForTensors}}  Proposition \ref{proposition:7.1} is derived from Proposition \ref{Proposition:2.1CauchySchwarzInequalityForTensors} simply by choosing an orthonormal basis in ${\mathbb C}^p$ with respect to which $\vec a$ becomes
the element $(0,\cdots,0,1)$ of ${\mathbb C}^p$.  On the other hand, Proposition \ref{Proposition:2.1CauchySchwarzInequalityForTensors} is derived from Proposition \ref{proposition:7.1} simply by choosing $r=p-1$ and $\vec a=(0,\cdots,0,1)\in{\mathbb C}^p$ and ${\mathbf S}=\vec a\wedge{\mathbf b}$ and ${\mathbf T}=\vec a\wedge{\mathbf c}$.

\begin{remark} In terms of the formulation of tensors with indices,
when the tensors ${\mathbf b}$ and ${\mathbf c}$ are represented by $p\times n$ matrices
$${\mathbf b}=\left(b_{\ell,k}\right)_{1\leq\ell\leq p,\,1\leq k\leq n}\quad{\rm and}\quad
{\mathbf c}=\left(c_{\ell,k}\right)_{1\leq\ell\leq p,\,1\leq k\leq n},
$$
the inequality in Proposition \ref{proposition:7.1} becomes
$$
\begin{aligned}&\left|\sum_{1\leq j<\ell\leq p}\sum_{k=1}^n
\left(a_j\,b_{\ell,k}-a_\ell\,b_{j,k}\right)\overline{\left(a_j\,c_{\ell,k}-a_\ell\,c_{j,k}\right)}
\right|^2\cr
&\leq q\left(\sum_{j=1}^p|a_j|^2\right)\sum_{\ell=1}^p\sum_{1\leq m<j\leq p}
\left|\sum_{k=1}^n\left(a_m\,b_{j,k}-a_j\,b_{m,k}\right)\overline{c_{\ell,k}}\right|^2,\cr
\end{aligned}
$$
because
$$
\vec a\wedge{\mathbf b}=\left(a_j\,b_{\ell,k}-a_\ell\,b_{j,k}\right)_{1\leq j<\ell\leq p,\,1\leq k\leq n}.
$$
Since
$$
\begin{aligned}&\sum_{1\leq j<\ell\leq p}\sum_{k=1}^n
\left(a_j\,b_{\ell,k}-a_\ell\,b_{j,k}\right)\overline{\left(a_j\,c_{\ell,k}-a_\ell\,c_{j,k}\right)}\cr
&=\sum_{1\leq j,\ell\leq p}\sum_{k=1}^n
\left(a_j\,b_{\ell,k}-a_\ell\,b_{j,k}\right)\overline{a_j}\,\overline{c_{\ell,k}}
\end{aligned}
$$
(from the fact that the inner product of a $2$-tensor $C$ with a skew-symmetric $2$-tensor $B$ is unchanged when
$C$ is replaced by its skew-symmetrization), by replacing $c_{\ell,k}$ by its complex-conjugate
we conclude that the inequality in Proposition \ref{proposition:7.1} is equivalent to
$$
\begin{aligned}&\left|\sum_{j,\ell=1}^p\sum_{k=1}^n
\overline{a_j}\left(a_j\,b_{\ell,k}-a_\ell\,b_{j,k}\right)c_{\ell,k}
\right|^2\cr
&\leq q\left(\sum_{j=1}^p|a_j|^2\right)\sum_{\ell=1}^p\sum_{1\leq m<j\leq p}
\left|\sum_{k=1}^n\left(a_m\,b_{j,k}-a_j\,b_{m,k}\right)c_{\ell,k}\right|^2,\cr
\end{aligned}
$$
which is precisely the inequality in Lemme 1 on p.552 of \cite{Skoda1972}.
\end{remark}

We give in Proposition \ref{Proposition:3.2NakanoCurvatureOfKernelSubbundleInNormalFiberCoordinate} the formula for the Nakano curvature of the kernel subbundle ${\mathcal K}$ in normal fiber coordinates.  Now for comparison between our proof and Skoda's original proof, we give the formula for it in general fiber coordinates as follows.

\begin{proposition}[Nakano Curvature of Kernel Subbundle]\label{proposition:7.4}  Let $g_1,\cdots,g_p$ be holomorphic functions on a domain $\Omega$ in ${\mathbb C}^n$ and let ${\mathcal K}$ be the kernel of the bundle-homomorphism $(g_1,\cdots,g_p):\left({\mathcal O}_\Omega\right)^{\oplus p}\to {\mathcal O}_\Omega$ which is given the metric induced from the standard flat metric of the trivial ${\mathbb C}$-vector bundle $\left({\mathcal O}_\Omega\right)^{\oplus p}$ of rank $p$.
The Nakano curvature of ${\mathcal K}$ is given by
$$
\begin{aligned}
\left(v^{j\lambda}\right)_{1\leq j\leq p, 1\leq\lambda\leq n}
&\mapsto\sum_{1\leq j,k\leq p,1\leq\lambda,\nu\leq n}\Theta_{j\bar k\lambda\bar\nu}v^{j\lambda}\overline{v^{k\nu}}
\cr
&=-\frac{1}{{\left(\sum_{j=1}^p|g_j|^2\right)^3}}\left|\sum_{j,\ell=1}^p\sum_{\lambda=1}^n\overline{g_\ell}
\left(g_\ell(\partial_\lambda g_j) -g_j(\partial_\lambda
g_\ell)\right)v^{j\lambda}\right|^2\cr
\end{aligned}
$$
for $\left(v^{j\lambda}\right)_{1\leq j\leq p, 1\leq\lambda\leq n}$ satisfying
$\sum_{j=1}^p g_jv^{j\lambda}=0$ for $1\leq\lambda\leq n$.
\end{proposition}

\begin{proof} Proposition \ref{Proposition:3.2NakanoCurvatureOfKernelSubbundleInNormalFiberCoordinate} is the special case where $g_1=\cdots=g_{p-1}=0$ and $g_p=1$ at the point under consideration. To prove Proposition \ref{proposition:7.4}, we can argue as
follows by applying a (constant) unitary transformation $U$ of order $p$ to the $p$-tuple of holomorphic functions $(g_1,\cdots,g_p)$.

Let
$\left(\tilde g_1,\cdots,\tilde g_p\right)$ be the image of $(g_1,\cdots,g_p)$ under the (constant) unitary transformation $U$ of order $p$.  For fixed $1\leq\lambda\leq n$ let
$\left(\tilde v^{1 \lambda},\cdots,\tilde v^{p \lambda}\right)\in{\mathbb C}^p$ be the image of $\left(v^{1 \lambda},\cdots,v^{p \lambda}\right)\in{\mathbb C}^n$ under the complex conjugate $\bar U$ of $U$ so that $\sum_{j=1}^p\tilde w_j\tilde v^{j \lambda}=\sum_{j=1}^p w_jv^{j \lambda}$ for any $(w_1,\cdots,w_p)\in{\mathbb C}$ whose image under $U$ is $\left(\tilde w_1,\cdots,\tilde w_p\right)$, because the image of
$\left(\overline{v^{1 \lambda}},\cdots,\overline{v^{p \lambda}}\right)$ under $U$ is $\left(\overline{\tilde v^{1 \lambda}},\cdots,\overline{\tilde v^{p \lambda}}\right)$.  Hence for fixed $1\leq\lambda\leq n$ we have
$$
\sum_{j,\ell=1}^p\overline{g_\ell}
\left(g_\ell(\partial_\lambda g_j) -g_j(\partial_\lambda
g_\ell)\right)v^{j \lambda}=\sum_{j,\ell=1}^p\overline{\tilde g_\ell}
\left(\tilde g_\ell(\partial_\lambda\tilde g_j) -\tilde g_j(\partial_\lambda
\tilde g_\ell)\right)\tilde v^{j \lambda}
$$
and summation over $1\leq\lambda\leq n$ yields
$$
\sum_{j,\ell=1}^p\sum_{\lambda=1}^n\overline{g_\ell}
\left(g_\ell(\partial_\lambda g_j) -g_j(\partial_\lambda
g_\ell)\right)v^{j \lambda}=\sum_{j,\ell=1}^p\sum_{\lambda=1}^n\overline{\tilde g_\ell}
\left(\tilde g_\ell(\partial_\lambda\tilde g_j) -\tilde g_j(\partial_\lambda
\tilde g_\ell)\right)\tilde v^{j \lambda}
$$
and
$$
\begin{aligned}&\frac{1}{\left(\sum_{j=1}^p|g_j|^2\right)^3}\left|\sum_{j,\ell=1}^p\sum_{\lambda=1}^n\overline{g_\ell}
\left(g_\ell(\partial_\lambda g_j) -g_j(\partial_\lambda
g_\ell)\right)v^{j\lambda}\right|^2\cr
&=
\frac{1}{\left(\sum_{j=1}^p|\tilde g_j|^2\right)^3}\left|\sum_{j,\ell=1}^p\sum_{\lambda=1}^n\overline{\tilde g_\ell}
\left(\tilde g_\ell(\partial_\lambda\tilde g_j) -\tilde g_j(\partial_\lambda
\tilde g_\ell)\right)\tilde v^{j\lambda}\right|^2.\cr
\end{aligned}
$$
Let $\tilde\Theta_{j\bar k\lambda\bar\nu}$ be the curvature tensor of the kernel vector subbundle when $(g_1,\cdots, g_p)$ is replaced by its image $(\tilde g_1,\cdots,\tilde g_p)$ under $U$.  Then
$$
\sum_{1\leq j,k\leq p,1\leq\lambda,\nu\leq n}\Theta_{j\bar k\lambda\bar\nu}v^{j\lambda}\overline{v^{k\nu}}
=\sum_{1\leq j,k\leq p,1\leq\lambda,\nu\leq n}\tilde\Theta_{j\bar k\lambda\bar\nu}\tilde v_{j\lambda}\overline{\tilde v^{k\nu}}
$$
for $\left(v^{k\lambda}\right)_{1\leq j\leq p,\,1\leq\lambda\leq n}$ satisfying $\sum_{j=1}^p g_jv^{j\lambda}=0$ for $1\leq\lambda\leq n$.
Hence at any prescribed point $P$ (where $g_1(P),\cdots,g_p(P)$ are not all zero) we can choose a (constant) unitary transformation $U$ of order $p$ so that the image $\left(\tilde g_1(P),\cdots,\tilde g_p(P)\right)$ of $(g_1(P),\cdots, g_p(P))$ satisfies the condition that $\tilde g_j(P)=0$ for $1\leq j\leq p-1$ and we can obtain the general case from the special case of $g_1(P)=\cdots=g_{p-1}(P)=0$ and $g_p(P)=1$.
\end{proof}

\begin{remark}  The arguments with inequalities centered around p.552 of \cite{Skoda1972} correspond to Proposition \ref{proposition:3.4.3} when the formula in Proposition \ref{proposition:7.4} is used.
\end{remark}

\bibliographystyle{amsalpha}

\end{document}